\documentclass[11pt, a4paper, reqno]{amsart}

\usepackage{pdflscape} 
\usepackage{hhline} 
\usepackage{array} 

\usepackage{calrsfs}
\DeclareMathAlphabet{\pazocal}{OMS}{zplm}{m}{n}

\usepackage{latexsym,amsmath,amsfonts,amscd,amssymb, mathrsfs, amsthm}
\usepackage[english]{babel}
\usepackage{appendix}
\usepackage{hyperref}
\usepackage{url}
\usepackage{mathtools}
\usepackage{booktabs}
\usepackage{csquotes}

\advance\oddsidemargin by -0.5cm
\advance\evensidemargin by -0.5cm
\advance\topmargin by -1cm 
\theoremstyle{plain}  
\newtheorem{theorem}{Theorem}[section]

\newtheorem*{theorem*}{Theorem}

\newtheorem{lemma}[theorem]{Lemma}
\newtheorem{proposition}[theorem]{Proposition}

\theoremstyle{definition}
\newtheorem{definition}[theorem]{Definition}

\newtheorem{remark}[theorem]{Remark}

\newtheorem*{claim*}{Claim}

\numberwithin{equation}{section}

\newcommand{\R}{\mathbb{R}}

\newcommand{\C}{\mathbb{C}}

\newcommand{\g}{\mathfrak{g}}
\newcommand{\h}{\mathfrak{h}}

\newcommand{\p}{\mathfrak{psh}}

\renewcommand{\P}{{\rm Psh}}

\newcommand{\ad}{\mathrm{ad}}

\newcommand{\Aut}{\mathrm{Aut}}

\usepackage{xcolor}
\definecolor{MyBlue}{RGB}{0,0,255}

\definecolor{MyRed}{RGB}{255,0,0}

\definecolor{MyGray}{RGB}{150,60,60}

\newcommand{\be}{\begin{equation}}
	\newcommand{\ee}{\end{equation}}

\newcommand{\ben}{\begin{enumerate}}
	\newcommand{\een}{\end{enumerate}}
\newcommand{\bit}{\begin{itemize}}
	\newcommand{\eit}{\end{itemize}}
\newcommand{\edoc}{\end{document}}

\makeatletter
\renewcommand{\tocsection}[3]{%
	\indentlabel{\@ifnotempty{#2}{\ignorespaces#1 #2\quad}}#3}
\renewcommand{\tocsubsection}[3]{%
	\indentlabel{\@ifnotempty{#2}{\ignorespaces#1 #2\quad}}#3}

\newcommand\@dotsep{4.5}
\def\@tocline#1#2#3#4#5#6#7{\relax
	\ifnum #1>\c@tocdepth 
	\else
	\par \addpenalty\@secpenalty\addvspace{#2}%
	\begingroup \hyphenpenalty\@M
	\@ifempty{#4}{%
		\@tempdima\csname r@tocindent\number#1\endcsname\relax
	}{%
		\@tempdima#4\relax
	}%
	\parindent\z@ \leftskip#3\relax \advance\leftskip\@tempdima\relax
	\rightskip\@pnumwidth plus1em \parfillskip-\@pnumwidth
	#5\leavevmode\hskip-\@tempdima{#6}\nobreak
	\leaders\hbox{$\m@th\mkern \@dotsep mu\hbox{.}\mkern \@dotsep mu$}\hfill
	\nobreak
	\hbox to\@pnumwidth{\@tocpagenum{\ifnum#1=1\fi#7}}\par
	\nobreak
	\endgroup
	\fi}
\AtBeginDocument{%
\expandafter\renewcommand\csname r@tocindent0\endcsname{0pt}
}
\def\l@subsection{\@tocline{2}{0pt}{2.5pc}{5pc}{}}
\makeatother

\begin{document}
\title{The completeness problem on the pseudo-homothetic Lie group}

\author[S. Chaib]{Salah Chaib}
\address{\hspace{-5mm} Salah Chaib, Centro de Matem\'{a}tica,
	Universidade do Minho,
	Campus de Gualtar,
	4710-057 Braga,
	Portugal} 
\email {salah.chaib@cmat.uminho.pt}
\author[A.C. Ferreira]{Ana Cristina Ferreira}
\address{\hspace{-5mm} Ana Cristina Ferreira, Centro de Matem\'{a}tica,
	Universidade do Minho,
	Campus de Gualtar,
	4710-057 Braga,
	Portugal} 
\email {anaferreira@math.uminho.pt}

\author[A. Zeghib]{Abdelghani Zeghib}
\address{\hspace{-5mm} Abdelghani Zeghib, UMPA, CNRS, \'Ecole Normale Sup\'erieure de Lyon, 46, All\'ee d'Italie 69364 Lyon Cedex 07, France }
\email{abdelghani.zeghib@ens-lyon.fr}

\date{\today}

\begin{abstract}
	Let us call pseudo-homothetic group the non-unimodular  3-dimensional Lie group that is the semi-direct product of $\R$ acting non-semisimply on $\R^2$.  In this article, we solve the geodesic completeness problem on this Lie group. In particular, we exhibit a family of complete metrics such that all geodesics have bounded velocity. As an application, we show that the set of complete metrics is not closed.
\end{abstract}

\maketitle

\vspace*{-3mm}

\tableofcontents

\section{Introduction}\label{sec:intro}
Any simply connected solvable 3-dimensional Lie group is a semi-direct product  $G_A$ of $\R^2$ by $\R$, where $A$ is a $2\times 2$ real matrix and $\R$ acts on $\R^2$, via $t \mapsto \exp t A$. Over $\C$, the matrix $A$ is always diagonalizable except in two cases, $B=\left(\begin{smallmatrix}0&1\\0&0\end{smallmatrix}\right)$ or $C=\left(\begin{smallmatrix} 1&1\\0&1\end{smallmatrix}\right)$, up to conjugacy and rescaling. Now, $G_B$ is the well-studied 3-dimensional Heisenberg group $\mathrm{Heis}$, and the group $G_C$ is the obbject of our study here. Its Lie algebra appears as type IV in the Bianchi classification and is generated  by a basis $B=\{e_1,e_2,e_3\}$   of $\R^3$ satisfying the bracket relations 
\begin{equation}\label{eq:LA-basis}
 [e_1,e_2] = e_2, \, [e_1, e_3] = e_2+e_3, \, [e_2,e_3] =0.   
\end{equation}
The conjugacy class of $C= \left(\begin{smallmatrix} 1&1\\0&1\end{smallmatrix}\right)$ accumulates to the identity matrix $I= \left(\begin{smallmatrix} 1&0\\0&1\end{smallmatrix}\right)$, and in a precise sense, 
$G_C$ accumulates to ${\rm Ho} = G_{I}$. The group ${\rm Ho}$ can be identified with the homothety group of the plane, that is,  
transformations $ z \in \C \mapsto a z + b$, $a \in \R^+, b \in \C$. Therefore, for these reasons, we shall call our Lie group $G_C$ the \emph{pseudo-homothety} group of dimension 3, and use the letters ${\rm Psh}$ for the group and $\mathfrak{psh}$ for its Lie algebra.

A matrix realization of the Lie algebra $\mathfrak{psh}$ is given by 

\begin{equation} \label{eq:LA-matrix-real}
\mathfrak{m}=\operatorname{span}\left\{E_1=\left(\begin{matrix} 0&0&0\\0&1&1\\0&0&1
\end{matrix}\right),E_2=\left(\begin{matrix} 0&0&0\\-1&0&0\\0&0&0
\end{matrix}\right),E_3=\left(\begin{matrix} 0&0&0\\-1&0&0\\-1&0&0
\end{matrix}\right)\right\},
\end{equation}
where the Lie bracket is the usual commutator of matrices. Indeed, we have $[E_1,E_2]=E_2$, $[E_1,E_3]=E_2+E_3$, and $[E_2,E_3]=0$, and, more precisely, \eqref{eq:LA-matrix-real} yields a linear representation of $\p$. Moreover, by making use of the exponential map, we obtain that the Lie group $\P$ is isomorphic to the matrix group 
\begin{equation}\label{eq:LG-matrix-real}
{\rm M}=\left\{\left(\begin{matrix}
	1&0&0\\-x_2-x_3& \mathrm{e}^{x_1}&x_1\mathrm{e}^{x_1}\\-x_3&0&\mathrm{e}^{x_1}
\end{matrix}\right) \ \hbox{ : }\, x_1,x_2,x_3\in\R        \right\},
\end{equation}
that is, ${\rm Psh}$ is isomorphic to $\mathbb{R}^3$,  with global coordinate system $(x_1,x_2,x_3)$, and multiplication given by the matrix multiplication of \eqref{eq:LG-matrix-real}. 

\smallskip

The present article aims to study left-invariant Lorentzian metrics on $\P$. This is the first part of a program aiming to understand geodesic completeness and isometry groups of left-invariant Lorentzian metrics on 3-dimensional non-unimodular Lie groups. The case of our present Lie group $\P$ seems, from the properties described below, to be of particular interest.

From the metric point of view,  ${\rm Heis}$ and ${\rm Ho}$ have antagonistic properties. For instance, any Lorentzian metric on ${\rm Heis}$ is complete \cite{Guediri-2step}, while, on the other hand, any Lorentzian metric on ${\rm Ho}$ is incomplete \cite{Guediri-solvable, VukmirovicSukilovic}. 
One might be tempted to think that our group $\P$ behaves like ${\rm Ho}$ from this point of view, that all of its Lorentzian metrics are incomplete, or at least, that complete metrics are rare.  Our main (somewhat surprising) result  is the following.

\begin{theorem}\label{(in)completeness of G} There is a one-parameter family of left-invariant Lorentzian complete metrics on $\P$  whose geodesics have bounded velocity. Furthermore, up to automorphism and scaling, there is a unique complete left-invariant Lorentzian metric on $\P$ with geodesics of unbounded velocity.
\end{theorem}

Having special metrics is, by any means, here by their completeness and  geodesic velocity (un)boundedness, an interesting phenomenon that deserves to be highlighted. 

\smallskip

The analysis of the geodesic flow for left-invariant metrics on Lie groups reduces, via the Euler-Arnold formalism, to the study of a quadratic homogeneous vector field on its Lie algebra. We shall call this vector field the \emph{geodesic field}, see Sec. \ref{sec:preliminaries} for the definition and some background and techniques.  

In Sec. \ref{Sec:N-F}, we will show that, under the action of $\R^\ast\times \mathrm{Aut}(\p)$, there are six equivalence classes of metrics on $\p$, two of them in families. Representatives of these equivalence classes are usually called normal forms. Completeness is constant on the orbits of this group action, and it is the first step to prove our main Th. \ref{(in)completeness of G},  cf. Sec. \ref{sec:geo-compl}. The geodesic fields of all metric normal forms can be found in Table \ref{table:geodesic-fields}, where we also summarized the following interesting facts. All geodesic fields have an invariant plane and, besides the energy, they have another  non-polynomial first integral, which is  defined on the complement of the invariant plane. In the complete case, the ``hidden'' non-trivial first integral was the key to the proof by establishing certain boundedness properties. Indeed, as stated in Th. \ref{(in)completeness of G}, there is a family of metrics whose integral curves of the geodesic field are all bounded; but, also remarkably, there is another complete metric with unbounded integral curves that happen to all lie on the invariant plane. 

\smallskip 

Up to covering and quotient, out of the six 3-dimensional unimodular Lie groups, there are only two that have incomplete metrics, $\mathrm{SL}(2,\mathbb{R})$ and $\mathrm{E}(1,1)$, as shown in \cite{BrombergMedina}. As it turns out, the set of complete metrics is closed for both of these groups. It was reasonable to conjecture that this was always the case, at least for 3-dimensional Lie groups. In Sec. \ref{sec:further-remarks}, we will see that our Lie group $\P$ provides a (non-unimodular) counter-example.

\begin{proposition}
 The set of complete metrics on $\P$ is neither open nor closed.
\end{proposition}

Kundt metrics have been intensively studied in general relativity and have attracted the interest of mathematicians in recent years. We would like to observe that the study of geodesic fields provides a natural context to investigate the existence of left-invariant Kundt metrics, as explained in Sec. \ref{sec:further-remarks}. We present a concise outline of the existence of Kundt structures on $\P$ and exhibit one that is complete and another one that is a plane wave.

\section{Preliminaries}\label{sec:preliminaries}

We include here a brief account of background material, for the sake of clearness of exposition, and also to fix notation and terminology. 

\subsection{The Euler-Arnold theorem} As is well-known, left-invariant metrics on a Lie group $G$ are in one-to-one correspondence with non-degenerate symmetric bilinear forms on its Lie algebra $\g$. The Euler-Arnold formalism allows us to treat questions concerning geodesics (here understood as the geodesics of the Levi-Civita connection) also at the Lie algebra level, as follows.

Let $I$ be an open interval in $\mathbb{R}$ and $\gamma: I \longrightarrow G$ be a smooth curve in $G$. Using left translations, we can define the associated curve $v: I \longrightarrow \g$ in the Lie algebra $\g$ of $G$, for every $t\in I$, as
\begin{equation*}
v(t) = D_{\gamma(t)}L_{\gamma^{-1}(t)}\dot{\gamma}(t).
\end{equation*}

Notice that for matrix Lie groups $v(t) = \gamma^{-1}(t)\dot{\gamma}(t)$.

We have the following theorem, first proved by Euler for the group $\mathrm{SO}(3)$, and then established in full generality by Arnold in his seminal work on applications of differential geometry of Lie groups to the hydrodynamics of perfect fluids. 

\begin{theorem}[Arnold, \cite{Arnold-paper, Arnold-book}]\label{Thm: EA-equation} Let $(G, q)$ be semi-Riemannian Lie group. The curve $\gamma: I \longrightarrow G$ is a geodesic if and only if the associated curve $v: I \longrightarrow \g$ satisfies, for every $t\in I$, the equation
\begin{equation}
\dot{v}(t) =  \mathrm{ad}_{v(t)}^\dagger v(t),\label{eq: EA-equation}
\end{equation}
where $\mathrm{ad}_{v(t)}^\dagger$denotes the formal adjoint of $\mathrm{ad}_{v(t)}$ with respect to $q$.
\end{theorem}

The system of ODE in \eqref{eq: EA-equation} is called the Euler-Arnold equation and its associated vector field in $\R^n$ is called the Euler-Arnold vector field. For simplicity of language, we will sometimes refer to this vector field as the geodesic field. Remark that the geodesic field is quadratic and homogeneous.

Recall that a vector field is said to be complete if all its integral curves have the real line $\R$ as maximal domain of definition, and incomplete otherwise. Clearly, a left-invariant metric on a Lie group is geodesically complete if and only if its associated Euler-Arnold vector field is complete. 

\subsection{First integrals} It is easy to see that $q(v,v)$ (sometimes referred to as the energy) is a first integral, that is, $q(v,v)$ is constant along any solution of \eqref{eq: EA-equation}. If $G$ can be equipped with a bi-invariant metric then another first integral is granted for every metric, \cite{BrombergMedina};  however, in general, there is no guarantee that another one exists.

\subsection{Idempotents} A technique that is very useful in the search for incomplete integral curves of quadratic homogeneous vector fields is that of idempotents.

\begin{definition}
Let $F$ be a quadratic homogeneous vector field on $\R^n$. A non-trivial solution of $F(v_o)=v_o$ is called an idempotent.
\end{definition}

It was proved in \cite{KaplanYorke} that for a quadratic homogeneous vector field, we can always find either a singularity (i.e. $F(v_o)=0$) or an idempotent. 

Moreover, as explained in \cite{BrombergMedina}, an idempotent $v_o$ yields an incomplete solution of the system $\dot{v}=F(v)$, since the solution with initial condition $v_o$ is given by $t\longmapsto u(t)v_o$, with $u$ such that $\dot{u}=u^2$ and $u(0)=1$.

\subsection{Incompleteness in dimension 1}

The ODE $\dot{u}=u^2$ is the typical prototype of an equation with incomplete solutions, the velocity of an integral curve grows quadratically and the curve reaches infinity in finite time. Heuristically, an ODE of the form $\dot{u}= u^2+\delta$ with $\delta>0$ should also be incomplete as the velocity grows even faster. We can formalize this statement with the following lemma. 

\begin{lemma}\label{incompleteness lemma}
	Let $(E)$ be an ordinary differential equation of the form
	$$
	\dot{x}(t)=ax^2(t)+\alpha(t),
	$$
	such that $a>0$ and $\alpha\in\mathcal{C}^\infty(\R)$; $t\mapsto\alpha(t)\geq0$. Let $\gamma\colon I\rightarrow\R$ be a nonzero maximal integral curve of $(E)$, then $\gamma$ must be incomplete.
\end{lemma}
\begin{proof}
	Suppose, aiming at a contradiction,  that $I=\R$. We start by assuming that $\gamma$ is bounded, i.e. there exists $M_1, M_2\in\R$ such that $M_1\leq\gamma(t)\leq M_2$.
	Since $\dot{\gamma}(t)=\gamma(t)^2+\alpha(t)\geq0$ then $\gamma$ is non-decreasing, and thus $\gamma$ has two horizontal asymptotes
	$$
	\lim\limits_{t\to -\infty}\gamma(t)=c_1 \ ,\ \lim\limits_{t\to +\infty}\gamma(t)=c_2, \qquad c_1,c_2\in \R,
	$$
	which, in turn, implies that $\lim\limits_{t\to \pm\infty}\dot{\gamma}(t)=0$. Hence, $\gamma\equiv0$ which contradicts the fact that $\gamma$ is nonzero maximal curve, therefore, $\gamma$ cannot be bounded. 
	Without loss of generality, we suppose that $\gamma$ is not upper-bounded, and we estimate the time it takes for $\gamma$ to tend from some $x_0=\gamma(t_0)>0$ to $+\infty$	
	$$
	\int\limits_{t_0}^{+\infty}dt=\int\limits_{x_0}^{+\infty}\frac{1}{\left(\frac{dx}{dt}\right)}dx=\int\limits_{x_0}^{+\infty}\frac{dx}{ax^2+\alpha}\leq\int\limits_{x_0}^{+\infty}\frac{dx}{ax^2}=\frac{1}{ax_0}<+\infty.
	$$
	Thus, $\gamma$ tends to infinity in finite time and is, therefore, an incomplete integral curve of $(E)$.
\end{proof}

\subsection{Action of the automorphism group}
Let $\mathrm{Sym}(\g)$ be the space of symmetric bilinear forms on $\g$ and $\mathrm{{Sym}^\ast(\g)}$ the subset of all non-degenerate ones. The automorphism group of the Lie algebra $\g$, 
\begin{equation*}
    \mathrm{Aut}(\g)=\{\varphi \in \mathrm{GL}(\g):~ [\varphi u, \varphi u] = \varphi[u,v]\, \,\, \forall u,v\in \g\}
\end{equation*}
acts on $\mathrm{Sym}^\ast(\g)$, as follows. Any $\varphi \in \mathrm{Aut}(\g)$ induces a map
$$\begin{array}{l}
     \mathrm{Sym}(\g) \longrightarrow \mathrm{Sym}(\g), \qquad m\longmapsto \varphi.m  \\
     \text{where } (\varphi.m)(u,v) = m(\varphi^{-1}u, \varphi^{-1}v), \quad \forall u, v \in \g, 
\end{array}$$
which naturally restricts to a map $\mathrm{Sym}^\ast(\g) \longrightarrow \mathrm{Sym}^\ast(\g)$.

Not too surprisingly, completeness of the flow of the geodesic field is invariant under rescaling and under the action of the automorphism group. The first statement is clear, the geodesic field remains unchanged by rescaling. The second was proved, for instance, in \cite{ElshafeiFerreiraSanchezZeghib}. Concretely, all semi-Riemannian metrics in each orbit of $\mathrm{Sym}^\ast(\g)$ by the action of $\mathrm{Aut}(\g)$ are either complete or incomplete. 

It is of interest to show that idempotents are also invariant under this action. More precisely, we have the following.

\begin{lemma}\label{lemma:idemp-aut}
    Let $m$ and $n$ be two elements in $\mathrm{Sym}^\ast(\g)$ such that $n=\varphi.m$ for some $\varphi \in \mathrm{Aut}(\g)$. Then $x_o$ is an idempotent of the geodesic field of $m$ if and only if $\varphi(x_o)$ is an idempotent of the geodesic field of $n$.      
\end{lemma}

\begin{proof}
Let $n=\varphi.m$ and let $\dagger_m$ and $\dagger_n$ denote the formal adjoints with respect to $m$ and $n$. Suppose that the geodesic field of $m$ has an idempotent $x_o$, i.e., $\ad_{x_o}^{\dagger_m}x_o=x_o$. Then $m(x_o, \ad_{x_o}y) = m(x_o,y)$, for all $y\in \g$, and so,  $n(\varphi(x_o), \varphi(\ad_{x_o} y))=n(\varphi(x_o), \varphi(y))$. Since $\varphi$ is an automorphism of $\g$, then $n(\varphi(x_o), \ad_{\varphi(x_o)}\varphi(y))=n(\varphi(x_o), \varphi(y))$ and, therefore, we have $n(\ad_{\varphi(x_o)}\varphi(x_0), \varphi(y)) = n(\varphi(x_o), \varphi(y))$. Hence, $\ad_{\varphi(x_o)}^{\dagger_n}\varphi(x_o) =\varphi(x_o)$. The converse is clear, since $n=\varphi^{-1}.m$.
\end{proof}

Representatives of the orbits of $\mathrm{Sym}^\ast(\g)$ under the action of $\R^*\times\mathrm{Aut}(\g)$, where $\R^*$ acts as scaling, are usually called metric normal forms.

\section{Normal forms of left-invariant metrics on {\P}}\label{Sec:N-F}

The classification of normal forms of left-invariant metrics in dimension 3 has been considered more or less implicitly in several articles, for instance \cite{BoucettaChakkar-m, HaLee23}. We include some details here for our Lie algebra $\p$, for clearness of exposition and illustration of the method. 

The automorphism group $\mathrm{Aut}(\p)$ can be obtained by direct computation using the definition and our preferred basis $B$ in \eqref{eq:LA-basis} as the matrix group
\begin{equation*}
 \mathrm{Aut}(\p) =   \left\{
        \begin{pmatrix}
            1 & 0 & 0 \\
            a & c & d \\
            b & 0 & c
        \end{pmatrix}: \, a,b,c,d \in \mathbb{R}, c \neq 0
    \right\}.
\end{equation*}
As can be seen, $\mathrm{Aut}(\p)$ is 4-dimensional and has two connected components. Consider the generic $3\times 3$ matrix  
\begin{equation*}
m= \begin{pmatrix}
    m_1 & m_2 & m_3 \\
    m_2 & m_4 & m_5 \\
    m_3 & m_5 & m_6
\end{pmatrix},
\end{equation*}
which is assumed to represent a non-degenerate symmetric bilinear form in the basis $B$. The image of $m$ under the automorphism $\varphi^{-1} = \left(\begin{smallmatrix}
    1 & 0 & 0 \\
    a & c & d \\
    b & 0 & c
\end{smallmatrix}\right)$ is given by $\varphi^tm\varphi$ where $t$ denotes the matrix transpose.

Observe that the restriction of $m$ to the derived subalgebra $\mathfrak{d}$ of $\p$ is transformed only by the
the subgroup $\{\varphi \in \mathrm{Aut}(\p): a = b = 0\}$.
Moreover, the (non)degeneracy on $\mathfrak{d}=\mathrm{span}\{e_2,e_3\}$ and  whether $e_2$ is isotropic or not are both preserved by $\Aut(\p)$.
We thus have two cases to consider, which will include subcases. 

Case 1: $m|_\mathfrak{d}$ is non-degenerate i.e. $m_4m_6 -m_5^2 \neq 0$.

Subcase 1.1: $e_2$ is non-isotropic i.e. $m_4 \neq 0$.

We have two possibilities here, which depend on the signs of both $m_4$ and the chosen scale (which in turn depends on the second and third leading principal minors).

\begin{center}

$\mathcal{Q}_{1,r} = \begin{pmatrix}
    1 & 0 & 0\\
    0 & 1 & 0 \\
    0 & 0 & r
    \end{pmatrix}$, \quad with $r \neq 0$ \qquad and \qquad  $\mathcal{Q}_{2,s} = \begin{pmatrix}
    1 & 0 & 0\\
    0 & -1 & 0 \\
    0 & 0 & s
    \end{pmatrix}$, \quad with $s \neq 0$.
\end{center}
Subcase 1.2: $e_2$ is isotropic i.e. $m_4 = 0$.

We also have two possibilities here, which depend on the signs of both $m_5$ and the chosen scale (which in turn depends on the second and third leading principal minors).

\begin{center}
$\mathcal{Q}_3 = \begin{pmatrix}
    1 & 0 & 0\\
    0 & 0 & 1 \\
    0 & 1 & 0
    \end{pmatrix}$ \qquad and \qquad  $\mathcal{Q}_4 = \begin{pmatrix}
    1 & 0 & 0\\
    0 & 0 & -1 \\
    0 & -1 & 0
    \end{pmatrix}$.
\end{center}

Case 2: $m|_\mathfrak{d}$ is degenerate i.e. $m_4m_6 -m_5^2 = 0$.

Case 2.1: $e_2$ is non-isotropic i.e. $m_4 \neq 0$.

\begin{center}
$\mathcal{Q}_5 = \begin{pmatrix}
    0 & 0 & 1\\
    0 & 1 & 0 \\
    1 & 0 & 0
    \end{pmatrix}$.
\end{center}

Case 2.2: $e_2$ is isotropic i.e. $m_4 = 0$.

\begin{center}
$\mathcal{Q}_6 = \begin{pmatrix}
    0 & 1 & 0\\
    1 & 0 & 0 \\
    0 & 0 & 1
    \end{pmatrix}$.
\end{center}

We remark that every Riemannian metric belongs to the orbit of $\mathcal{Q}_{1,r}$ for some $r>0$. 

\section{Euler-Arnold vector field of left-invariant metrics on {\P} }

For each of the normal forms in Sec. \ref{Sec:N-F}, we exhibit its corresponding geodesic field as well as some extra properties.

\subsection{The geodesic vector field}

Let $v(t)=x(t)e_1+y(t)e_2+z(t)e_3$ be a curve on $\p$ equipped with a quadratic form $q$. The geodesic system of ODEs is then 
$$\dot{v}= \ad_v^\dagger v,$$ which can be readily computed by using the fact that $\ad_v^\dagger = Q^{-1}\ad_v^tQ$, where $Q$ is the matrix of $q$ and the superscript $t$ represents the matrix transpose.

For instance for $\mathcal{Q}_3$, we can easily compute that the geodesic field is given by the following system of ODEs
\begin{equation*}
\mathcal{F}_3=\begin{cases}
	\dot{x}=-2yz-z^2\\
	\dot{y}=x(y+z)\\ 
	\dot{z}=xz
\end{cases}.     
\end{equation*}
Similar computations will allow us to obtain the geodesic field for every normal form of Sec. \ref{Sec:N-F}, see Table \ref{table:geodesic-fields}.

\subsection{First integrals}

As expected, the energy $e(x,y,z) = x^2+2xz$ is a quadratic first integral of $\mathcal{F}_3$. However, no other quadratic first integrals exist. This can be shown by direct computation, by parametrizing all possible polynomials of degree at most 2 in the variables $x,y,z$. Nevertheless, a non-quadratic partially defined first integral can be found. 

\begin{proposition}
	In the subspace of $\p$ given by $\{z\neq0\}$, the following expression is an invariant of the geodesic field of $\mathcal{Q}_3$
	$$
	f(x,y,z)=\ln|z|-\frac{y}{z}.
	$$
	In other words, $f$ is a first integral of the geodesic field  $\mathcal{F}_3$ restricted to $\{z\neq0\}$ which is invariant since $\{z=0\}$ is. 
\end{proposition}
\begin{proof} Clearly, $\{z=0\}$ is an invariant plane of $\mathcal{F}_3$. It suffices to   show that the total time derivative of $f$ is zero on $\{z\neq0\}$: 
		$$
		\frac{d}{dt}f=\frac{\dot{z}(t)}{z(t)}-\frac{\dot{y}(t)z(t)-\dot{z}(t)y(t)}{z(t)^2}=x(t)-x(t)=0.
		$$
		The proposition, thus, follows.
	\end{proof}

Interestingly, this property is not exclusive of $\mathcal{Q}_3$. Analogous computations will show that all normal forms have an invariant plane and a non-quadratic partially defined first integral on its invariant complement, see Table \ref{table:geodesic-fields}.

\subsection{Normal forms}

The following table organizes the information discussed in the previous two subsections for all metric normal forms of $\p$.
{\small
\begin{table}[h]
\begin{tabular}{*5c} \toprule  
    & \quad bilinear form \quad  &  \quad geodesic field \quad &  invariant plane   &  \quad first integrals \quad \\ \toprule 
 $\mathcal{Q}_{1,r\neq 0}$  &  $\begin{pmatrix}
     1 & 0 & 0 \\
     0 & 1 & 0 \\ 0 & 0 & r \end{pmatrix}$ & $\begin{cases}
	\dot{x}=-y^2-rz^2-yz\\
	\dot{y}=xy\\ 
	\dot{z}=xz+\frac{1}{r}xy
\end{cases}$ & $y=0$ & \begin{tabular}{c}
     $x^2+ y^2 +rz^2$  \\ \smallskip 
     $\ln |y| - r\dfrac{z}{y}$ 
\end{tabular} \\ \midrule  
      $\mathcal{Q}_{2,s\neq 0}$  &  $\begin{pmatrix}
     1 & 0 & 0 \\
     0 & -1 & 0 \\ 0 & 0 & s \end{pmatrix}$ & $\begin{cases}
	\dot{x}=y^2-sz^2+yz\\
	\dot{y}=xy\\ 
	\dot{z}=xz-\frac{1}{s}xy
\end{cases}$ & $y=0$ & \begin{tabular}{c}
     $x^2- y^2 +sz^2$  \\ \smallskip 
     $\ln |y| + s\dfrac{z}{y}$ 
\end{tabular} \\ \midrule  
 $\mathcal{Q}_{3}$  &  $\begin{pmatrix}
     1 & 0 & 0 \\
     0 & 0 & 1 \\ 0 & 1 & 0 \end{pmatrix}$ & $\begin{cases}
	\dot{x}=-2yz-z^2\\
	\dot{y}=x(y+z)\\ 
	\dot{z}=xz
\end{cases}$ & $z=0$ & \begin{tabular}{c}
     $x^2+ 2yz$  \\ \smallskip 
     $\ln |z| - \dfrac{y}{z}$ 
\end{tabular} \\ \midrule  
 $\mathcal{Q}_{4}$  &  $\begin{pmatrix}
     1 & 0 & 0 \\
     0 & 0 & -1 \\ 0 & -1 & 0 \end{pmatrix}$ & $\begin{cases}
	\dot{x}=2yz+z^2\\
	\dot{y}=x(y+z)\\ 
	\dot{z}=xz
\end{cases}$ & $z=0$ & \begin{tabular}{c}
     $x^2-2yz$  \\ \smallskip 
     $\ln |z| - \dfrac{y}{z}$ 
\end{tabular} \\ \midrule  
 $\mathcal{Q}_{5}$  &  $\begin{pmatrix}
     0 & 0 & 1 \\
     0 & 1 & 0 \\ 1 & 0 & 0 \end{pmatrix}$ & $\begin{cases}
	\dot{x}=x^2+xy\\
	\dot{y}=xy\\ 
	\dot{z}=-y^2-(x+y)z
\end{cases}$ & $y=0$ & \begin{tabular}{c}
     $y^2 +2xz$  \\ \smallskip 
     $\ln |y| - \dfrac{x}{y}$ 
\end{tabular} \\ \midrule  
 $\mathcal{Q}_{6}$  &  $\begin{pmatrix}
     0 & 1 & 0 \\
     1 & 0 & 0 \\ 0 & 0 & 1 \end{pmatrix}$ & $\begin{cases}
	\dot{x}=x^2\\
	\dot{y}=-x(y+z)-z^2\\ 
	\dot{z}=x(x+z)
\end{cases}$ & $x=0$ & \begin{tabular}{c}
     $z^2+2xy$  \\ \smallskip 
     $\ln |x| - \dfrac{z}{x}$ 
\end{tabular} \\ \bottomrule 
\end{tabular}   \medskip
\caption{Normal forms of geodesic vector fields on $\p$ and their first integrals.}\label{table:geodesic-fields}
\end{table}
}

\section{Geodesic (in)completeness of {\P}}\label{sec:geo-compl}

The aim of this section is to give the classification of geodesic completeness for all left-invariant metrics on $\P$. As previously discussed in Sec. \ref{sec:preliminaries}, it suffices (although this is by no means a trivial matter) to analyze the completeness of the flow for each of the geodesic vector fields in Table \ref{table:geodesic-fields}.

In what follows, we will denote by $\mathcal{F}_k$ the geodesic vector field associated to the bilinear, symmetric, nondegenerate form $\mathcal{Q}_k$, for every possible subscript $k$ listed in Sec. \ref{Sec:N-F}.

\subsection{Incomplete metrics}

It was shown in \cite{BrombergMedina} that for a  3-dimensional  unimodular Lie algebra, the Euler-Arnold vector field of a Lorentzian metric is incomplete if and only if it admits an idempotent; however, a counter-example in the non-unimodular case was given for a Lie algebra of Bianchi type VI. Our Lie algebra $\p$, while having some of its geodesics fields with idempotents, also provides such counter-examples. It is important to observe here that if a metric has no idempotents, then no other idempotents can exist in the same orbit, cf. Lemma \ref{lemma:idemp-aut}.

\subsubsection{Incomplete metrics with idempotents} 

\begin{enumerate}
    \item[]  $\mathcal{Q}_{1,r}$ with $r<0$: 
    $v_o= \left(1, 0, \frac{1}{\sqrt{-r}}\right)$ is an idempotent of $\mathcal{F}_{1,r}$, for $r<0$.
    \item[] $\mathcal{Q}_{2,s}$ with $s<0$: $v_o= \left(1, 0, \frac{1}{\sqrt{-s}}\right)$ is an idempotent of $\mathcal{F}_{2,s}$, for $s<0$.
    \item[] $\mathcal{Q}_5$: $v_0=(1,0,0)$ is an idempotent of $\mathcal{F}_5$. 
\end{enumerate}

\subsubsection{Incomplete metrics with no idempotents}

\begin{enumerate}
    \item[] $\mathcal{Q}_4$: The null integral curves of $\mathcal{F}_4$ satisfy the equation $2yz=x^2$. Replacing this on the first equation of $\mathcal{F}_4$ we get the equation $\dot{x}=x^2+z^2$. By Lemma \ref{incompleteness lemma}, incomplete integral curve exist.
    \item[] $\mathcal{Q}_6$: One of the ODEs of $\mathcal{F}_6$ is $\dot{x}=x^2$. Then, incomplete integral curves exist.
\end{enumerate}

\subsection{Complete metrics}

\subsubsection{Completeness of the family of metrics $\mathcal{Q}_{2,s}$ with $s>0$} \hfill 

Let $s>0$ and denote by $q_s$ the quadratic Lorentzian form on $\p$ associated to $\mathcal{Q}_{2,s}$, i.e. $q_s(x,y,z)=x^2-y^2+sz^2$.  
By inspecting the geodesic field $\mathcal{F}_{2,s}$, in Table \ref{table:geodesic-fields}, we see that the plane $\{y=0\}$ is an invariant spacelike plane. The curves with initial condition $(x_o, y_o, z_o)$, with $y_o=0$, satisfy the equation $x^2+sz^2=c_o$, where $c_o=x_o^2+sz_o^2 \geq 0$.  Thus, such integral curves are bounded and are, therefore, complete.

Observe that the geodesic field $\mathcal{F}_{2,s}$ is invariant under the the involution $(\mathrm{id}, -\mathrm{id}, -\mathrm{id})$. This means that if $\gamma(t)=(x(t),y(t),z(t))$ is the maximal integral curve with initial conditions $(x_o,y_o,z_o)$, then $\Tilde{\gamma}(t)=(x(t), -y(t), -z(t))$ is the maximal integral  curve with initial conditions $(x_o, -y_o, -z_o)$. Therefore, it is enough to analyze the behavior of the integral curves in the upper-half space $\{y>0\}$. 

Recall, from Table \ref{table:geodesic-fields}, that $\mathcal{F}_{2,s}$ has another non-quadratic first integral, defined for $y>0$ by $h_s(x,y,z)= \ln(y)+s \tfrac{z}{y}$. Therefore, any integral curve $\gamma(t)=(x(t),y(t),z(t))$ of $\mathcal{F}_{2,s}$ with $y(0)=y_o>0$ will be supported in the intersection of two level sets of $q_s$ and $h_s$, that is 
\begin{equation*}
    \begin{cases}
    x(t)^2-y(t)^2+sz(t)^2 = k \\
    \ln y(t)+ s \tfrac{z(t)}{y(t)} = c
    \end{cases}, \qquad \text{ where } k,c\in \R. 
\end{equation*}
From the first equation above, we see that $x(t)^2+sz(t)^2=k+y(t)^2$, which implies that $x(t)$ and $z(t)$ will be bounded when $y(t)$ is. Also, since $x(t)^2=k+y(t)^2-sz(t)^2$, then $sz(t)^2-y(t)^2 \leq k$. From the second equation, $z(t) = \frac{1}{s}(c-\ln y(t))y(t)$. Therefore,
$$\frac{y(t)^2}{s}((c-\ln y(t))^2-s) \leq k.$$
Since $\frac{y^2}{s}((c-\ln y)^2-s)$ tends to $+\infty$ when $y$ tends to $+\infty$, we conclude that $y(t)$ is bounded, otherwise we obtain a contradiction with the inequality above. 

Summing up, the integral curves of $\mathcal{F}_{2,s}$ are bounded which yields completeness of the metric $\mathcal{Q}_{2,s}$, $s>0$.

\subsubsection{Completeness of the metric $\mathcal{Q}_3$} \hfill 

The analysis of this case is very similar to the previous one, with the main difference that there are unbounded (complete) integral curves of the geodesic field $\mathcal{F}_3$.

It can be readily checked that $\{z=0\}$ is a lightlike (i.e. degenerate) invariant plane and that the maximal solution of $\mathcal{F}_3$ with initial condition $(x_o, y_o, 0)$ is given by $\gamma(t) = (x_o, y_o\mathrm{exp}(x_ot), 0)$. These curves are complete and unbounded. The involution $(\mathrm{id}, -\mathrm{id}, -\mathrm{id})$ leaves the geodesic field invariant and, therefore, it suffices to analyse the upper-half space $\{z>0\}$.  

From Table \ref{table:geodesic-fields}, we see that  $q(x,y,z)=x^2+2yz$ and $h(x,y,z) = \ln(z) - \tfrac{y}{z}$ are two first integrals of $\mathcal{F}_3$. Therefore, any integral curve $\gamma(t)=(x(t),y(t),z(t))$ of $\mathcal{F}_{3}$ with $z(0)=z_o>0$ will be supported in the intersection of two level sets of $q$ and $h$, that is 
\begin{equation*}
    \begin{cases}
    x(t)^2+2y(t)z(t) = k \\
    \ln z(t) - \tfrac{y(t)}{z(t)} = c
    \end{cases}, \qquad \text{ where } k,c\in \R. 
\end{equation*}
We will now show, as in the previous case, that these two first integrals imply the boundedness of the integral curves in $\{z>0\}$.  Let $\gamma(t)=(x(t), y(t), z(t))$ and suppose that $z(t)$ is bounded. Since $y(t)=z(t)(\ln z(t) - c)$, then $y(t)$ is also bounded (remark that even if $z(t)$ approaches zero, $y(t)$ remains bounded since $\lim_{z\rightarrow 0} z \ln z =0$). Also, $x(t)^2= k-2y(t)z(t)$, thus $x(t)$ is also bounded since $y(t)$ and $z(t)$ are. It remains then to show that $z(t)$ is necessarily bounded. We have that $2y(t)z(t)\leq k$ and thus $2z(t)^2(\ln z(t) -c) \leq k$.  This inequality implies that $z(t)$ is bounded. The proof that $\mathcal{Q}_3$ is complete follows.

\subsubsection{A dynamical study of the geodesic field of $\mathcal{Q}_3$} \hfill 

We wish to include in our discussion on the flow of the ODE system $\mathcal{F}_3 \colon \dot{v}=F(v)$ the following observations on its dynamics.

The vector field $\mathcal{F}_3$ admits three singular directions, which correspond to the zeros of $F$. 
We denote them as $w_1=(1,0,0)$, $w_2=(0,1,0)$, and $w_3=\left(0,-\tfrac{1}{2},1\right)$. 
This means that, for any scalar $\mu\in\R$, $F(\mu\, w_i)=0$ for all $i\in\{1,2,3\}$.

The eigenvalues of $D_{\lambda w_i}F$, with $i\in\{1,2,3\}$ and $\lambda\in\R^*$, provide key insights into the dynamics of the geodesic field  $\mathcal{F}_3$. Specifically:

\begin{enumerate}
	\item For $\alpha>0$, the singularity $\alpha W_1$ has one zero eigenvalue and two equal eigenvalues $\alpha$. This implies that the singularity is repelling.
	Thus, this singularity corresponds to time $-\infty$ for the corresponding integral curves in the invariant plane $\{z=0\}$ and in the associated spacelike level $\{x^2+2yz=\alpha^2\}$.
	
	\item For $\alpha<0$, the singularity $\alpha w_1$ again has one zero eigenvalue and two equal eigenvalues $\alpha$, meaning that it is attracting.
	Therefore, this singularity corresponds to time $+\infty$ for the corresponding integral curves in the invariant plane $\{z=0\}$ and the associated spacelike level $\{x^2+2yz=\alpha^2\}$.
	
	\item For any $\beta\in\R$, the singularity $\beta w_2$ has all eigenvalues equal to zero. 
	As a result, the vector field is parallel to the line given by $w_2$, which is the $y$-axis.
	
	\item For any $\delta\in\R^*$, the singularity $\delta w_3$ has one zero eigenvalue, and the remaining two are $\pm\delta \sqrt{2}i$, causing an elliptic behavior (see Figure \ref{fig:vf-singular-line}).
	This elliptic nature of the linearisation of $F$ at this singularity explains why timelike integral curves are periodic and rotate around the corresponding singularity.
\end{enumerate}

\begin{figure}[h]
	\centering
	\includegraphics[width=130mm]{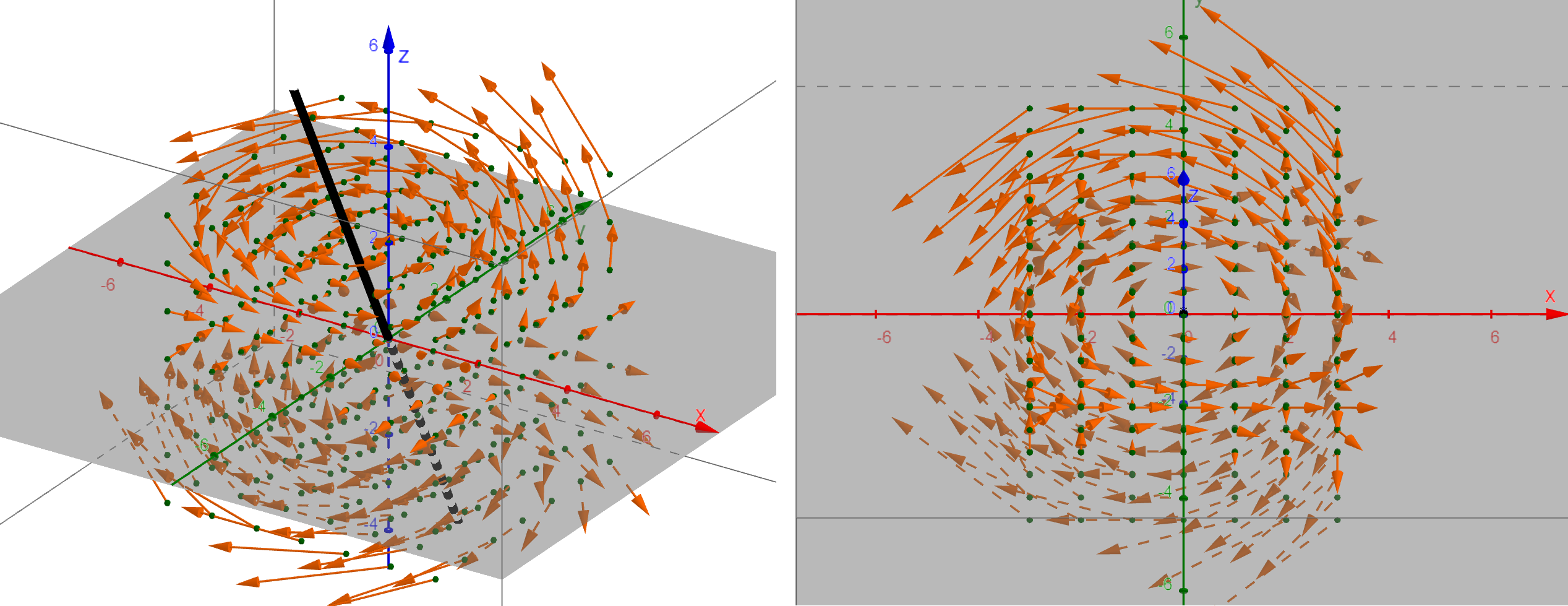}
	\caption{{\small The geodesic vector field $\mathcal{F}_3$ qualitatively rotates around the singular line (shown in black) defined by the singularity $w_3$, with the direction determined the sign of the $z$-coordinate.}}
	\label{fig:vf-singular-line}
\end{figure}

\begin{remark}
	Notably, $D_\cdot F$ has always a zero eigenvalue for every singularity. This zero eigenvalue corresponds to the normal direction of the surfaces defined by the corresponding energy level set at the singularity in question.
\end{remark}

\section{Further remarks}\label{sec:further-remarks}

\subsection{(Non-)closedness of complete metrics}

We know, from \cite{BrombergMedina}, that there are, up to covering and quotient, only two 3-dimensional unimodular Lie groups with incomplete metrics. They are $\mathrm{SL}(2,\mathbb{R})$, the special linear group of degree 2, and $\mathrm{E}(1,1)$ the group of motions of Minkowski 2-space, also known as $\mathrm{Sol}$. Careful reading of the information in \cite[Props. 3 and 5]{BrombergMedina} allows us to conclude that the set of complete metrics in both $\mathrm{SL}(2,\R)$ and $\mathrm{E}(1,1)$ is closed. See also \cite{ElshafeiFerreiraReis} for a detailed description on $\mathrm{SL}(2,\R)$.

We will now show that this is not the case for our Lie group $\P$. In fact, the set of complete metrics is neither closed nor open. Fix $r\neq 0$, for every $n\in \mathbb{N}$, the matrix 
\begin{equation*}
    A_{n,r}= \begin{pmatrix}
        1 & 0 & 0 \\
        0 & \frac{1}{n^2} & 1 \\
        0 &  1 & \frac{r}{n^2}+n^2
    \end{pmatrix}
\end{equation*}
is an element in the in the orbit $\mathcal{Q}_{1,r}$. Taking $r=-n^4 <0$, the sequence $A_{n,-n^4}$ converges to $\mathcal{Q}_3$ as $n$ tends to $+\infty$. Hence, we have a sequence of incomplete metrics converging to a complete one, showing that the setting of incomplete metrics is not closed, and thus, the set of complete metrics is not open. Now, fix $s\neq 0$, for every $n\in \mathbb{N}$, the matrix 
\begin{equation*}
    B_{n,s}= \begin{pmatrix}
        1 & 0 & 0 \\
        0 & -\frac{1}{n^2} & -1 \\
        0 & - 1 & \frac{s}{n^2}-n^2
    \end{pmatrix}
\end{equation*}
is an element in the orbit $\mathcal{Q}_{2,s}$. Taking $s=n^4 >0$, the sequence $B_{n,n^2}$ converges to $\mathcal{Q}_4$ as $n$ tends to $+\infty$. Therefore, we have a sequence of complete metrics converging to an incomplete one, showing that the set of complete metrics is not closed. 

\subsection{Kundt metrics} A Lorentzian manifold $(M,g)$ is said to be a \emph{Kundt spacetime} if there exists
a non-singular vector field $V$ on $M$ and a differential one-form $\alpha$ such that
\begin{equation}
    g(V,V) = 0, \quad \nabla_X V=\alpha(X)V, \quad \nabla_V V = 0,\label{eq:kundt}
\end{equation}
for any vector field $X$ orthogonal to $V$. In 
 \cite{BoucettaMelianiZeghib}, the following definition was introduced in order to provide an algebraic characterization of the Kundt property for left-invariant structures.

 \begin{definition}
 Let $\g$ be a Lie algebra. A \emph{Kundt pair} on $\g$ is a pair $(\langle -, - \rangle, \h)$, where $\langle -, - \rangle$ is a Lorentzian inner product on $\g$ and $\h$ is a degenerate codimension one subalgebra which is stable by the Levi-Civita product $\bullet$ and such that for any $e \in \h^\perp$, $e\bullet e =0$.
 \end{definition}

In \cite[Prop. 3.1]{BoucettaMelianiZeghib}, it was, indeed, proved that a Lie group whose Lie algebra has a Kundt pair is a Kundt Lie group, that is, a Lie group with a left-invariant Lorentzian metric and a left-invariant vector field $V$ satisfying the definition of a Kundt spacetime, cf. \eqref{eq:kundt}.

 As mentioned in Sec.\,\ref{sec:intro}, the study of geodesic fields provides a natural context to investigate the existence of Kundt metrics, as it is not difficult to see that having a degenerate subalgebra which is stable by the Levi-Civita product is equivalent to having a degenerate subalgebra which is an invariant plane for the corresponding geodesic field.  

A quick check of Table \ref{table:geodesic-fields}   shows that, if $\mathfrak{f}=\mathrm{span}\{e_1,e_2\}$ and $\mathfrak{d}=\mathrm{span}\{e_2,e_3\}$, then we have the following Kundt pairs  $(\mathcal{Q}_3,\mathfrak{f})$, $(\mathcal{Q}_4, \mathfrak{f})$, $(\mathcal{Q}_5, \mathfrak{d})$, $(\mathcal{Q}_6, \mathfrak{d})$. Notice that $(\mathcal{Q}_3,\mathfrak{f})$ is a complete Kundt structure on $\p$.

We remark that $\mathcal{Q}_6$ is flat and, in particular, it is a plane wave. In the global coordinates given by \eqref{eq:LG-matrix-real},  a frame of left-invariant vector fields is given by $\{X_1,X_2,X_3\}$, with $X_1 = \partial_{x_1}$, $X_2=\mathrm{e}^{x_1}\partial_{x_2}$, and $X_3=\mathrm{e}^{x_1}(x_1\partial_{x_2}+\partial_{x_3})$. Thus, our left-invariant metric $\mathcal{Q}_6$ is expressed on $\P$, in this coordinate system, as
$${g} = \mathrm{e}^{2x1}dx_3^2 + \mathrm{e}^{x_1}(dx_1dx_2+x_1dx_1dx_3).$$
In the paper \cite{CFZ-isometry}, another part of our program mentioned in Sec. \ref{sec:intro}, we show that this is the only Lorentzian metric on $\P$  with 4-dimensional isometry group.

\section*{Acknowledgments}  

{\small
\begingroup
\sloppy

The authors would like to thank Souheib Allout, Lilia Mehidi, and Henrik Winther for useful discussions and valuable comments. Thanks are also due to Mauro Mantegazza for help with a typesetting issue.

The first and second named authors acknowledge the support of CMAT (Centro de Matem\'atica da Universidade do Minho). Their research was financed by Portuguese Funds through FCT (Fundação para a Ciência e a Tecnologia, I.P.) within the Projects UIDB/00013/2020, UIDP/00013/2020, UID/00013: Centro de Matemática da Universidade do Minho (CMAT/UM) and also within the doctoral grant UI/BD/154255/2022 of the first named author. 

 The first named author also acknowledges partial support from LABEX MILYON (ANR-10-LABX-0070) of Université de Lyon, within the framework of the “France 2030” program (ANR-11-IDEX-0007), managed by the French National Research Agency (ANR).

\endgroup
}


\begin{thebibliography}{19}


\bibitem{Arnold-paper}
    {Arnold, V.I.},
   \emph{Sur la g\'eom\'etrie diff\'erentielle des groupes de {L}ie de
              dimension infinie et ses applications \`a l'hydrodynamique
           des fluides parfaits},
 {Ann. Inst. Fourier (Grenoble)}, {16}, {319--361},
     {1966}.
    
  \bibitem{Arnold-book}
    Arnold, V. I., \emph{Mathematical methods of classical mechanics},
    {Graduate Texts in Mathematics},
    {60},
   {Springer-Verlag, New York},
     {1989}.
     
    
    \bibitem{BrombergMedina}
    {Bromberg, S. and Medina, A.},
     \emph{Geodesically complete {L}orentzian metrics on some homogeneous 3 manifolds}, {SIGMA Symmetry Integrability Geom. Methods Appl.}, {4:13}, {Paper 088}, {2008}.
     
\bibitem{BoucettaChakkar-m}
    {Boucetta, M. and Chakkar, A.},
     \emph{The moduli spaces of {L}orentzian left-invariant metrics on
              three-dimensional unimodular simply connected {L}ie groups}, {J. Korean Math. Soc.}, {59:4},{651--684}, {2022}.
  
   
\bibitem{BoucettaMelianiZeghib}
    {Boucetta, M., Meliani, A. and Zeghib, A.},
     \emph{Kundt three-dimensional left invariant spacetimes},
   {J. Math. Phys.}, {63:11}, {Paper No. 112501, 14}, {2022}.
     
\bibitem{CFZ-isometry}
    {Chaib, S., Ferreira, A.C and Zeghib, A.},
    \emph{Isometries of 3-dimensional semi-Riemannian Lie groups},
{arXiv:2502.17193}, 2025.

\bibitem{ElshafeiFerreiraReis}
 {Elshafei, A. and Ferreira, A.C. and Reis, H.},
 \emph{Geodesic completeness of pseudo and holomorphic-{Riemannian} metrics on {Lie} groups}, {Nonlinear Anal., Theory Methods Appl., Ser. A, Theory Methods}, {232}, {Id/No 113252},
{2023}.


\bibitem{ElshafeiFerreiraSanchezZeghib}
   {Elshafei, A., Ferreira, A.C., S\'anchez,
              M. and Zeghib, A.},
     \emph{Lie groups with all left-invariant semi-{R}iemannian metrics complete}, {Trans. Amer. Math. Soc.}, {377:8}, {5837--5862}, {2024}.
  
\bibitem{Guediri-2step}
   {Guediri, M.},
     \emph{Sur la compl\'etude des pseudo-m\'etriques invariantes a
              gauche sur les groupes de {L}ie nilpotents},
   {Rend. Sem. Mat. Univ. Politec. Torino},
             {52:4},
     {371--376}, {1994}.
     
\bibitem{Guediri-solvable}
 {Guediri, M.},
 \emph{On completeness of left-invariant {Lorentz} metrics on solvable {Lie} groups},
 {Rev. Mat. Univ. Complutense Madr.},
 {9:2},
{337--350},
{1996}.     
        
\bibitem{HaLee23}
    {Ha, K.Y. and Lee, J.B.},
     \emph{Left invariant {L}orentzian metrics and curvatures on
              non-unimodular {L}ie groups of dimension three},
{J. Korean Math. Soc.}, {60:1}, {143--165}, {2023}.

\bibitem{KaplanYorke}
    {Kaplan, J. L. and Yorke, J. A.},
     \emph{Nonassociative, real algebras and quadratic differential
              equations},
    {Nonlinear Anal.},
 {3:1}, {49--51}, {1979}.
 
 \bibitem{VukmirovicSukilovic}
    {Vukmirovi\'c, S. and  Sukilovi\'c, T.},
     \emph{Geodesic completeness of the left-invariant metrics on {$\mathbb{R}H^n$}}, {Ukra\"in. Mat. Zh.},
   {72:5}, {611--619},
    {2020}.
  

\end{thebibliography}


\end{document}